\documentclass[hidelinks]{article}
\usepackage{graphicx} 
\usepackage{hyperref}
\usepackage{amsthm}
\usepackage{geometry}
\usepackage{amssymb}
\usepackage{amsmath}
\usepackage{color}
\definecolor{purple}{rgb}{0.5,0,0.5}
\usepackage[T1]{fontenc}
\usepackage{lmodern}
\usepackage[english]{babel}

\newtheorem{theorem}{Theorem}[section]
\newtheorem{lemma}[theorem]{Lemma}

\newtheorem{claim}[theorem]{Claim}

\newtheorem{conjecture}[theorem]{Conjecture}

\theoremstyle{definition}

\newcommand{\claimproofstart}[1][Proof]{\begin{proof}[#1]
\renewcommand{\qedsymbol}{$\boxdot$}}
\newcommand\claimproofend{
\end{proof}
\renewcommand{\qedsymbol}{$\square$}}

\newcommand{\cA}{\mathcal{A}}

\newcommand{\cH}{\mathcal{H}}

\newcommand{\cG}{\mathcal{G}}
\newcommand{\cM}{\mathcal{M}}
\newcommand{\cF}{\mathcal{F}}
\newcommand{\cE}{\mathcal{E}}

\newcommand{\Var}{\mathrm{Var}}
\newcommand{\Inf}{\mathrm{Inf}}

\title{Probabilistic and game thresholds}

\author{Rajko Nenadov \footnote{School of Mathematics and Statistics, University of Canterbury, New Zealand. Research supported by the New Zealand Marsden Fund. Email: rajko.nenadov@canterbury.ac.nz}} 
\date{}

\begin{document}

\maketitle

\begin{abstract}
    Let $X$ be a finite set and $\cF \subseteq 2^X$ a non-trivial increasing property. We show that the bias-threshold $b(\cF)$ for the Picker-Chooser $\cF$-game satisfies $b(\cF) \ge \alpha / p_c(\cF)$, where $p_c(\cF)$ is the threshold for the property $\cF$ and $\alpha > 0$ is an absolute constant. This proves that Picker-Chooser games satisfy one-sided ``probabilistic intuition''. The proof is based on the notion of influence from the theory of boolean functions.
\end{abstract}

\section{Introduction}
Positional games are a large class of two-player games with full information. A game is described by a (usually finite) set $X$, called a \emph{board} (we shall assume $|X| \ge 2$ throughout the paper); a family $\cF \subseteq 2^X$, where $2^X$ denotes the power set of $X$, whose members are called \emph{winning sets}; a mechanism for selecting elements of $X$; and a rule determining which player wins the game. Starting with the papers of Hales and Jewett \cite{hales63} and Erd\H{o}s and Selfridge \cite{erdos73game}, these games have attracted wide attention. For a thorough introduction to the topic, see the monographs by Beck \cite{beck08book} and Hefetz, Krivelevich, Stojakovi\'c and Szab\'o \cite{hefetz14book}.

We study the following game mechanism, introduced by Beck \cite{beck02picker} under the name of \emph{Picker-Chooser} games (and more recently studied under the name \emph{Waiter-Client}). For a given \emph{bias} parameter $b \in \mathbb{N}$, the game proceeds in rounds as follows: In each round the first player picks $b$ previously unclaimed elements of $X$, and the second player chooses one of them and adds it to the initially empty set $S \subseteq X$. The remaining $b - 1$ elements are discarded from the game. The game finishes when there are fewer than $b$ elements remaining in $X$. The first player wins if, at the end of the game, it holds that $S \in \cF$. Note that this is a natural game analogue of the Achlioptas process \cite{achlioptas09science}, where in each round $b$ new elements are chosen uniformly at random and the player has to choose one. For some results about this game, see \cite{bednarska13weight,bednarska16manipulative,bednarsa14picker,hefetz16wc,hefetz17wcham,nenadov23probabilistic}.

It is clear that if $\cF$ is \emph{increasing}, in the sense that if $A \in \cF$ and $A \subseteq A' \subseteq X$ then $A' \in \cF$, then the game is monotone in $b$. That is, if the first player has a winning strategy with bias $b$, then they also have a winning strategy when playing with bias $1 \le b' < b$. Therefore, for every non-trivial increasing property $\cF$ we can define the \emph{threshold-bias} $b(\cF)$ as the largest value of $b$ for which the first player has a winning strategy. The study of biased positional games is largely concerned with determining such a threshold for various game mechanisms and goals.

Notice that, at the end of the game, $|S| = \lfloor |X| / b \rfloor$. A common heuristic for predicting which player has the winning strategy, known as the ``probabilistic intuition'', is to look at whether a \emph{random} set $S \subseteq X$ of density $1/b$ is likely to belong to $\cF$ or not. Our main result, Theorem \ref{thm:main}, essentially proves that one direction of this heuristic is always correct for Picker-Chooser games: If a random set of density $1/b$ is likely to belong to $\cF$, then Picker has a winning strategy after reducing the bias by a suitable absolute constant factor. As we shall discuss later, the converse does not always hold.

To state the result precisely, we introduce some notation. For $p \in [0,1]$, let $\mu_p$ be the product measure on $2^X$ given by $\mu_p(A) = p^{|A|} (1 - p)^{|X \setminus A|}$. A set $A \sim \mu_p$, in expectation, has density precisely $p$. We shall assume throughout that $\cF \subseteq 2^X$ is increasing and \emph{non-trivial}, that is, $\emptyset \not \in \cF$ and $X \in \cF$. It is well-known that in this case $\mu_p(\cF) = \sum_{A \in \cF} \mu_p(A)$ is continuous in $p$ and there exists a unique value $p \in [0,1]$ such that $\mu_p(\cF) = 1/2$. We call such a value of $p$ the \emph{threshold} and denote it by $p_c(\cF)$. The following is the main result of the paper.

\begin{theorem} \label{thm:main}
Let $X$ be a finite set and $\cF \subseteq 2^X$ a non-trivial increasing property. Then 
$$
    b(\cF) \ge (128 p_c(\cF))^{-1}.
$$
\end{theorem}   

The main appeal of Theorem \ref{thm:main} is that it applies with no restrictions on a family $\cF$ (other than the most natural ones), and it allows one to study a significantly easier probabilistic problem in order to determine a lower bound on the bias-threshold. Indeed, determining $p_c(\cF)$ for various properties is one of the central topics in probabilistic combinatorics. 

Before we give examples of applications of Theorem \ref{thm:main} and discuss its optimality, we first give a \emph{first-moment} upper bound on $b(\cF)$.

\paragraph{First-moment upper bound.}

We say that $\cG \subseteq 2^X$ \emph{covers} $\cF$ if for every $A \in \cF$ there exists $B \in \cG$ such that $B \subseteq A$. A family  $\cF$ is $q$-small, for $q \in [0,1]$, if there exists $\cG$ which covers $\cF$ and
$$
    \sum_{T \in \cG} q^{|T|} \le 1/2.
$$
 Such a family $\cG$ provides a first-moment bound on the probability that $A \sim \mu_q$ belongs to $\cF$. The \emph{expectation-threshold} $q(\cF)$ is then defined as the largest value of $q$ such that $\cF$ is $q$-small. 
 
 It is clear that $q(\cF) \le p_c(\cF)$, and the Kahn-Kalai conjecture \cite{kahn07threshold} states that $p_c(\cF) \le C q(\cF) \log(|X|)$, for an absolute constant $C > 0$. Building on the advances in the Sunflower Conjecture \cite{alweiss21sunflower} and a fractional version of the Kahn-Kalai conjecture \cite{frankston21fractional}, this was recently resolved by Park and Pham \cite{park24expectation}. With this in mind, Theorem \ref{thm:main} implies
$$
    b(\cF) \ge \alpha / (q(\cF) \log(|X|)),
$$
for an absolute constant $\alpha > 0$. The following follows from the Erd\H{o}s-Selfridge potential method \cite{erdos73game}; for completeness, we include a proof.

\begin{lemma} \label{lemma:upper_bound}
    Let $X$ be a finite set and $\cF \subseteq 2^X$ a non-trivial increasing property. Then 
    $$
        b(\cF) < q(\cF)^{-1}.
    $$
\end{lemma}
\begin{proof}
    Let $b \ge q(\cF)^{-1}$ be an integer, and choose a cover $\cG$ of $\cF$ such that $\sum_{T \in \cG} q(\cF)^{|T|} \le 1/2$. If $S$ is the current set chosen by the second player and $D$ is the set of discarded elements, consider the potential
    $$
        \Phi(S,D) = \sum_{\substack{T \in \cG\\ T \cap D = \emptyset}} b^{-|T \setminus S|}.
    $$
    Initially, $\Phi(\emptyset, \emptyset) \le 1/2$. For any offer $Q$ of size $b$, average the potential after the round over the $b$ possible choices. The contribution of a set $T$ disjoint from $D$ is unchanged if $|T \cap Q|=0$, unchanged on average if $|T \cap Q|=1$, and zero if $|T \cap Q|\ge2$. Thus the second player can always choose an element for which the potential does not increase. If $S \in \cF$, some $T \in \cG$ satisfies $T \subseteq S$ and contributes $1$, a contradiction. Hence the second player wins for every integer $b \ge q(\cF)^{-1}$.
\end{proof}

The previous lemma, together with the resolution of the Kahn-Kalai conjecture, shows that the bound in Theorem \ref{thm:main} is always optimal up to $O(\log |X|)$ factor. 

\subsection{Applications}

Estimating thresholds for various properties is a central topic of probabilistic combinatorics. This provides us with a plethora of interesting applications of Theorem \ref{thm:main}.

\paragraph{Powers of Hamilton cycles.} Let $X$ denote the edges of $K_n$, the complete graph with $n$ vertices, and let $\cH^{(k)}$ denote the property of containing the \emph{$k$-th power} of a Hamilton cycle, for some fixed $k \in \mathbb{N}$. This is defined as the cycle on $n$ vertices, where we additionally put an edge between any two vertices at distance at most $k$. Powers of Hamilton cycles were the topic of Seymour's conjecture \cite{seymour74problem}, famously resolved for sufficiently large graphs by Koml\'os, S\'ark\"ozy and Szemer\'edi \cite{komlos98seymour} using the blow-up lemma and the regularity method. It follows from a general result of Riordan \cite{kuhn12powers,riordan00spanning} that $p_c(\cH^{(k)}) = O_k(n^{-1/k})$ in the case $k \ge 3$. The case $k = 2$ turned out to be more challenging, and it was only obtained recently by Kahn, Narayanan, and Park \cite{kahn21square} that $p_c(\cH^{(2)}) = O(n^{-1/2})$. One easily computes $q(\cH^{(k)}) = \Omega_k(n^{-1/k})$, which provides the matching lower bound. Therefore, Theorem \ref{thm:main} and Lemma \ref{lemma:upper_bound} imply the following. 

\begin{theorem}
    For every integer $k \ge 2$ there exist $c, C > 0$ such that the following holds. Let $n$ be sufficiently large, let $X$ denote the edges of $K_n$ and let $\cH^{(k)}$ denote the property of containing the $k$-th power of a Hamilton cycle. Then 
    $$
        c n^{1/k} \le b(\cH^{(k)}) \le C n^{1/k}.
    $$
\end{theorem}

The reason why this is interesting is that the corresponding thresholds for $k \ge 2$ are proved using genuine probabilistic ideas, rather than through combinatorial means. In particular, they provide no insight into how to find a power of a Hamilton cycle, but merely show that one is likely to exist. It should then come as a surprise that, without any knowledge about how these powers are created in the random setting, one can transfer the threshold to an entirely deterministic setting.

\paragraph{Perfect matchings in hypergraphs.}

Let now $X$ denote the edges of $K_n^{(3)}$, the complete $3$-uniform hypergraph with $n$ vertices. Let $\cM$ denote the property of containing a matching of maximum possible size, that is, the property that there exists a subset of $\lfloor n/3 \rfloor$ vertex-disjoint hyperedges (a perfect matching when $3 \mid n$). Determining the threshold $p_c(\cM)$ was a notoriously difficult problem, known as Shamir's problem \cite{schmidt83matching}. The problem was finally solved by Johansson, Kahn, and Vu \cite{johansson08factor}, who confirmed that $p_c(\cM) = \Theta(\log(n) / n^2)$. It is remarkable that this result also follows now from the resolution of the Kahn-Kalai conjecture. The very reason why the Kahn-Kalai conjecture was considered to be a difficult problem, other than the generality in which it is stated, is precisely because it would imply a solution to Shamir's problem.

As in the proof of the threshold for powers of Hamilton cycles, proofs of the threshold for $\cM$ are entirely probabilistic in nature. This is even more emphasised here. In the case of powers of Hamilton cycles, there are combinatorial proofs \cite{fischer22triangle,nenadov19powers} which show that one can find these structures at densities of the form $(\log n)^{O(1)} / n^{1/k}$. These arguments are based on the ideas of \emph{absorbers} and, while they do employ probabilistic techniques, they do so largely to establish a particular local combinatorial property which is then used to find the desired structure. To the best of the author's knowledge, the only combinatorial arguments for finding perfect matchings in random hypergraphs are based on Hall's matching theorem and require density $p$ of order at least $1 / n$, which is significantly higher than $p_c(\cM)$. Similarly as in the case of powers of Hamilton cycles, we find it rather surprising that the threshold for perfect matchings in the random setting can be transferred to the deterministic setting of Picker-Chooser games.

\begin{theorem}
    There exist absolute constants $c,C>0$ such that the following holds for all sufficiently large $n$. Let $X$ denote the edges of $K_n^{(3)}$, and let $\cM$ denote the property of containing $\lfloor n/3 \rfloor$ vertex-disjoint edges. Then
    $$
        c n^2 / \log n \le b(\cM) \le C n^2.
    $$
\end{theorem}

The upper bound in the previous result comes from a straightforward estimate of $q(\cM)$. It remains a challenging problem to determine the order of $b(\cM)$.

\paragraph{Hamilton cycles.}

As remarked earlier, Lemma \ref{lemma:upper_bound} and the resolution of the Kahn-Kalai conjecture imply that Theorem \ref{thm:main} is off by a factor at most $O(\log(|X|))$. We now describe an example where it is indeed off by this much. Namely, let $X$ again denote the edges of $K_n$, and this time let $\cH$ denote the property of containing a Hamilton cycle. It is a classical result of P\'osa \cite{posa76hamilton} that $p_c(\cH) = \Theta(\log(n) / n)$, and so Theorem \ref{thm:main} implies $b(\cH) \ge \alpha n / \log(n)$ for an absolute constant $\alpha$. However, a result of Bednarska-Bzd{\k{e}}ga, Hefetz, Krivelevich and \L uczak \cite{bednarska16manipulative} shows $b(\cH) = \Omega(n)$, which matches the bound $q(\cH)^{-1} = O(n)$.

\subsection{Further remarks}

It remains an interesting question whether one can further refine the lower bound in Theorem \ref{thm:main}. With the previous examples in mind, the following seems plausible. If true, this would give a precise characterisation of the bias threshold.

\begin{conjecture}
    Let $X$ be a finite set and $\cF \subseteq 2^X$ a non-trivial increasing property. Then
    $$
        b(\cF) \ge \alpha q(\cF)^{-1},
    $$
    for some absolute constant $\alpha > 0$.
\end{conjecture}

\paragraph{Motivation.} Finally, let us briefly discuss the author's motivation for studying this question. Resolving a conjecture of  Bednarska-Bzd{\k{e}}ga, Hefetz, and {\L}uczak \cite{bednarsa14picker}, the author \cite{nenadov23probabilistic} determined the bias-threshold for the Picker-Chooser $H$-graph game. More precisely, let $X$ correspond to the edges of $K_n$ and let $\cF_H$ denote the property of containing $H$ as a subgraph. The result from \cite{nenadov23probabilistic} states that there exists $\alpha_H > 0$, depending on $H$, such that
$$
    b(\cF_H) \ge \alpha_H n^{1/m(H)}
$$
where $m(H)$ denotes the maximum density of $H$, and is defined as
$$
    m(H) = \max \left\{ \frac{e(H')}{v(H')} \; : \; H' \subseteq H, e(H') \ge 1 \right\}.
$$
By the result of Bollob\'as \cite{bollobas81smallsubgraphs}, the same bound on $b(\cF_H)$ is now also implied by Theorem \ref{thm:main} (likely with a much better constant factor). The corresponding upper bound, up to the constant $\alpha_H$, again follows from $q(\cF_H) = \Omega_H(n^{-1/m(H)})$. 

The proof from \cite{nenadov23probabilistic} proceeds by building copies inductively on the number of edges of $H$. As a side-product, it provides (though this has not been written anywhere) an alternative proof of a Janson-type bound on the probability that the random graph $G(n,p)$ contains a copy of $H$. Somewhat far-fetched, the author's hope was that transferring a result about general increasing properties to the Picker-Chooser setting would shed some light on the probability of the random subset having the property. While some insight has been gained, namely Lemma \ref{lemma:main}, this remains wide open.

\paragraph{Use of AI.} The proof was developed with assistance from ChatGPT Astra. The author takes responsibility for the mathematical claims and references. AI assistance was also used in revising the exposition.

\section{Preliminaries}

Even though the game is entirely deterministic, we shall make use of the probabilistic method and estimates. The main machinery in our proof is based on the idea of \emph{influence} from the theory of Boolean functions. In this section we lay out the necessary background.

\paragraph{Probabilistic estimates.} All logarithms are natural. Let $Y$ be a non-negative random variable and $\mathcal{E}$ an event of positive probability in some probability space. Then
\begin{equation} \label{eq:exp_cond}
    \mathbb{E}[Y \mid \mathcal{E}] \le \frac{\mathbb{E}[Y]}{\Pr(\cE)}.
\end{equation}

For $p \in [0,1]$ and a finite set $U$, let $U_p$ denote the product probability distribution over $\{0, 1\}^U$ of Bernoulli($p$) random variables. For $C \subseteq U$ and $x \in \{0,1\}^U$, let $Z_C(x) = \sum_{i \in C} x_i$, and let $\cA \subseteq \{0,1\}^U$ such that
$$
    a := \Pr_{\xi \sim U_p}[\xi \in \cA] > 0.
$$
Then, for $\xi = (\xi_i)_{i \in U} \sim U_p$, we have
\begin{equation} \label{eq:condition}
    \mathbb{E} \left[ Z_C(\xi) \mid \xi \in \cA \right] \le \ln(1 / a) + (e-1) |C| p.
\end{equation}
This bound follows from
$$
    e^{\mathbb{E}[Z_C(\xi) \mid \xi \in \cA]} \le 
    \mathbb{E}\left[e^{Z_C(\xi)} \mid \xi \in \cA \right] \stackrel{\eqref{eq:exp_cond}}{\le} 
    \frac{\mathbb{E}\left[e^{Z_C(\xi)} \right]}{a} \le 
    \frac{1}{a} (1 + (e-1)p)^{|C|},
$$
where the first inequality is an application of Jensen's inequality. Applying $\ln(1 + x) \le x$ for every $x > -1$ yields \eqref{eq:condition}.

\paragraph{Boolean functions.} We state some basic facts from the theory of Boolean functions. For a thorough introduction, see the monograph by O'Donnell \cite{donnell2014analysis}. 

Let $U$ be a finite set and $f : \{0, 1\}^U \to \{0,1\}$ an increasing function. For $i \in U$, the ``discrete derivative'' $D_i f(x)$ is defined as
$$
    D_i f(x) = f(x^{i \to 1}) - f(x^{i \to 0}),
$$
where $x^{i \to k}$ is obtained from $x$ by setting the value of $x_i$ to $k$. The \emph{$p$-influence} of $i$ is then defined as
$$
    \Inf_i(f) = \mathbb{E} \left[ D_i f(\xi) \right] = \Pr\left[ D_i f(\xi) = 1 \right],
$$
where $\xi \sim U_p$. In other words, $\Inf_i(f)$ is the probability that changing the $i$-th coordinate also changes the outcome of the function. For $0<p<1$, by the law of total expectation, we have
$$
    \mathbb{E}[ f(\xi) ] = p \; \mathbb{E}[f(\xi) \mid \xi_i = 1] + (1 - p) \; \mathbb{E}[f(\xi) \mid \xi_i = 0].
$$
This implies
\begin{equation} \label{eq:X_1}
    \mathbb{E} \left[ f(\xi) \mid \xi_i = 1 \right] - \mathbb{E}[f(\xi)] = (1 - p) \Inf_i(f)
\end{equation}
and
\begin{equation} \label{eq:X_0}
    \mathbb{E}[f(\xi)] - \mathbb{E}\left[ f(\xi) \mid \xi_i = 0 \right]  = p \; \Inf_i(f).
\end{equation}

The following is a variant of the discrete Poincar\'e inequality (e.g.~see \cite[Chapter 8, Propositions 8.23--8.24]{donnell2014analysis}) that we shall use here.

\begin{lemma}
    Let $U$ be a finite non-empty set and $f : \{0,1\}^U \to \{0,1\}$ an increasing function. For $p \in [0,1]$ and $\xi \sim U_p$, we have
    \begin{equation} \label{eq:var_influence}
        \Var[f(\xi)] \le p(1-p) \sum_{i \in U} \Inf_i(f).
    \end{equation}
\end{lemma}
\begin{proof}
  We prove the lemma by induction on the size of $U$. For $|U| = 1$, the claim is immediate. Let $q \in U$ and set $U' = U \setminus \{q\}$. Suppose that the claim holds for $U'$ and any increasing Boolean function on $U'$. 
  
  For $t\in\{0,1\}$ and $x \in \{0,1\}^{U'}$, define $f_t(x)=f(x,t)$. Note that both functions are increasing. Let $\xi' \sim U'_p$ and $\xi_q \sim \textrm{Bernoulli}(p)$ be independent, and note that $\xi = (\xi', \xi_q) \sim U_p$. By the law of total variance,
  \begin{align*}
      \Var[f(\xi)] &= (1-p)\Var[f_0(\xi')] + p\Var[f_1(\xi')] + p(1-p)
      \bigl(\mathbb{E}[f_1(\xi')]
            -\mathbb{E}[f_0(\xi')]\bigr)^2 \\
      &= (1-p)\Var[f_0(\xi')] + p\Var[f_1(\xi')] + p(1-p) \; \Inf_q(f)^2.
  \end{align*}
  For every $i \in U'$, independence and the law of total expectation give
  \begin{align*}  
      \Inf_i(f) &= (1-p)\mathbb{E}[D_i f_0(\xi')] + p\mathbb{E}[D_i f_1(\xi')] \\
      &= (1 - p) \Inf_i(f_0) + p\; \Inf_i(f_1).
  \end{align*}
  Applying the induction hypothesis to $f_0$ and $f_1$, we obtain
  $$
      \Var[f(\xi)] \le p(1-p)\left( \sum_{i \in U'}\Inf_i(f) + \Inf_q(f)^2 \right)
      \le p(1-p)\sum_{i \in U} \Inf_i(f),
  $$
  where the final inequality uses $0 \le \Inf_q (f) \le1$.
  \end{proof}

\section{Proof}

The following lemma is the key ingredient of the proof of Theorem \ref{thm:main}.

\begin{lemma} \label{lemma:main}
    Let $U$ be a finite set and $f : \{0, 1\}^U \to \{0, 1\}$ an increasing function. Let $p \in (0,1)$ be such that $\mathbb{E}[f(\xi)] = 1/2$ for $\xi \sim U_p$. Then, for every integer $1 \le b \le (64 p)^{-1}$, there exists a set $Q \subseteq U$ of size $|Q| = b$ such that, for every $j \in Q$,
    $$
        \mathbb{E} \left[ f(\xi) \mid  \xi_j = 1 \text{ and } \xi_i = 0 \text{ for $i \in Q \setminus \{j\}$} \right] \ge 1/2.
    $$    
\end{lemma} 

The proof of Lemma \ref{lemma:main} is based on the notion of influence and equalities \eqref{eq:X_1} and \eqref{eq:X_0}. To this end, suppose we have a set $Q \subseteq U$ of size $b$ such that $\Inf_i(f) \approx \Inf_j(f)$ for every two $i, j \in Q$. Setting $\xi_j := 1$, by \eqref{eq:X_1}, causes the gain 
$$
    \mathbb{E}[f(\xi) \mid \xi_j = 1] - \mathbb{E}[f(\xi)] = (1 - p) \; \Inf_j(f) \approx \Inf_j(f).
$$
Setting $\xi_i := 0$ for $i \in Q \setminus \{j\}$, by \eqref{eq:X_0}, causes the loss
$$
    \mathbb{E}[f(\xi)] - \mathbb{E}[f(\xi) \mid \xi_i = 0] = p\; \Inf_i(f) \approx p \; \Inf_j(f).
$$
Therefore, we expect the total loss to be of the form
$$
    (b-1) \; p  \;\Inf_j(f) < \Inf_j(f).
$$
Since $bp \le 1 / 64$, we expect the total loss due to $\xi_i = 0$ for $i \in Q \setminus \{j\}$ to be smaller than the gain coming from $\xi_j = 1$. 

The problem in the previous arguments is that the conditioning in determining gain/loss has to be considered as a whole rather than individually. In particular, we want the conditioned gain not to drop too much, that is, we want
$$
    \mathbb{E}[f(\xi) \mid \xi_j = 1 \text{ and } \xi_i = 0 \text{ for $i \in Q \setminus \{j\}$}] - \mathbb{E}[f(\xi) \mid \xi_i = 0 \text{ for $i \in Q \setminus \{j\}$}] \ge c \;\Inf_j(f),
$$
for some absolute constant $c$. By using the fact that $bp$ can be made arbitrarily small, this suffices to ensure that the conditioned gain is still larger than the total loss. Bounding the total conditioned loss, on the other hand, is simpler as we only seek an upper bound, so \eqref{eq:exp_cond} can be used.

\begin{proof}[Proof of Lemma \ref{lemma:main}]
    Throughout the proof, we let $\xi \sim U_p$. If $b=1$, choose a coordinate of positive influence, which exists because $f$ is non-constant, and apply \eqref{eq:X_1}. We may therefore assume $b\ge2$. Since $\mathbb{E}[f(\xi)] = 1/2$ and $f$ is a Boolean function, we have $\Var[f(\xi)] = 1/4$. By \eqref{eq:var_influence}, we have 
    \begin{equation} \label{eq:contra}
        1/4 \le p(1-p) \sum_{i \in U} \Inf_i(f).
    \end{equation}
    For an integer $\ell \ge 1$, let $B_\ell = \{ i \in U \; : \; 2^{-\ell} < \Inf_i(f) \le 2^{-\ell + 1}\}$. Note that for some $\ell \ge1$ we then have $|B_\ell| \ge 4b \ell$. If this were not the case, we would have
    $$
        \sum_{i \in U} \Inf_i(f) \le \sum_{\ell=1}^{\infty} 2^{-\ell+1} |B_\ell| \le 4b \sum_{\ell=1}^{\infty} 2^{-\ell + 1}\ell = 16 b.
    $$
    As $b \le (64 p)^{-1}$, this would imply 
    $$
        p(1-p) \sum_{i \in U} \Inf_i(f) < p \cdot 16b \le 1/4,
    $$
    which contradicts \eqref{eq:contra}. 
    
    Consider some $\ell \ge 1$ such that $|B_\ell| \ge 4b \ell$, and set $B=B_\ell$. The following claim proves the key step in the outlined strategy, namely that the conditioned influence does not drop too much compared to the unconditioned one. For brevity, for a set $S \subseteq U$ we use the notation $\xi_S = 0$ to denote $\xi_i = 0$ for every $i \in S$.

    \begin{claim} \label{claim:key}
        There exists a set $Q \subseteq B_\ell$ of size $|Q| = b$ such that for every $j \in Q$ we have
        $$
            \mathbb{E}[D_j f(\xi) \mid \xi_{Q \setminus \{j\}} = 0] \ge \frac{\Inf_j(f)}{8 (1 - p)^{b-1}}.
        $$
    \end{claim}
    \claimproofstart
        For $j \in B$, the bound \eqref{eq:condition} gives
        $$
            \mathbb{E} \left[Z_{B \setminus \{j\}}(\xi) \mid D_j f(\xi) = 1\right] \le \ell \log(2) + (e-1)|B|p.
        $$
        The key is to show that there exists $Q \subseteq U$ of size $|Q| = b$ such that, for every $j \in Q$, we have
        \begin{equation} \label{eq:Q}
            \Pr\left[ \xi_{Q \setminus \{j\}} = 0 \mid D_j f(\xi) = 1 \right] \ge 1/8.
        \end{equation}
        Indeed, once we have such a set, the claimed bound follows from Bayes' theorem and \eqref{eq:Q}: For any $j \in Q$ and $Q' = Q \setminus \{j\}$, we have
        $$
            \Pr[D_j f(\xi) = 1 \mid \xi_{Q'} = 0] = \frac{\Pr[ \xi_{Q'} = 0 \mid D_j f(\xi) = 1] \cdot \Pr[D_j f(\xi) = 1]}{\Pr[\xi_{Q'} = 0]} \stackrel{\eqref{eq:Q}}{\ge} \frac{\Inf_j(f)}{8 (1 - p)^{b-1}}.
        $$
        
        We now prove that there exists a set $Q$ which satisfies \eqref{eq:Q}. Consider some $j \in B$. For a set $L' \subseteq B \setminus \{j\}$ of size $|L'| = 2b - 1$, chosen uniformly at random, set
        $$
            A_j(L') := \mathbb{E}_{\xi}\left[ Z_{L'}(\xi) \mid D_j f(\xi) = 1 \right].
        $$
        Here we write $\xi$ in the subscript to signify that the expectation is over $\xi \sim U_p$. Note that $A_j(L')$ is indeed a random variable which depends only on $L'$ and not on $\xi$. By the linearity of expectation and symmetry,
        \begin{align*}
            \mathbb{E}_{L'}[A_j(L')] &= \frac{2b - 1}{|B| - 1} \sum_{i \in B \setminus \{j\}} \mathbb{E}_{\xi}[ \xi_i \mid D_j f(\xi) = 1] \\
            &= \frac{2b - 1}{|B| - 1} \mathbb{E}_{\xi}[ Z_{B \setminus \{j\}}(\xi) \mid D_j f(\xi) = 1 ] \\
            &\le \frac{2b - 1}{|B| - 1} (\ell \log(2) + (e-1)|B|p) < \frac{7}{16}.
        \end{align*}
        By Markov's inequality, we thus have $\Pr_{L'} [ A_j(L') > 7/8 ] \le 1/2$. 
        
        Consider now a set $L \subseteq B$ of size $|L| = 2b$, chosen uniformly at random. For $j \in B$, let $I_j$ be the indicator random variable for the event that $j \in L$ and $A_j(L \setminus \{j\}) \le 7/8$. Conditioned on $j\in L$, the set $L\setminus\{j\}$ is a uniformly random $(2b-1)$-subset of $B\setminus\{j\}$. Since $\Pr(j\in L)=2b/|B|$, we have $\mathbb{E}_{L}[I_j] \ge (2b/|B|)(1/2)=b/|B|$. Therefore, by the linearity of expectation we get
        $$
            \mathbb{E}_L \Bigl[ \sum_{j \in B} I_j \Bigr] \ge b.
        $$        
        In particular, there exists a choice of a $(2b)$-set $L$ such that for at least $b$ elements $j \in L$ it holds that $A_j(L \setminus \{j\}) \le 7/8$. Let $Q \subseteq L$ denote an arbitrary set of size $b$ consisting of such elements, and note that, for every $j\in Q$, Markov's inequality gives
        $$
            \Pr[\xi_{Q\setminus\{j\}}\ne0\mid D_j f(\xi)=1]
            \le \mathbb{E}[Z_{Q\setminus\{j\}}(\xi)\mid D_j f(\xi)=1]
            \le A_j(L\setminus\{j\})\le7/8.
        $$
        Thus $Q$ satisfies \eqref{eq:Q}.
    \claimproofend

    The previous claim provides a lower bound on the conditional gain coming from $\xi_j = 1$. We now bound the loss coming from setting all other coordinates to 0. To this end, let $j \in Q$ and let $(q_1, \ldots, q_{b-1})$ be an arbitrary ordering of the elements in $Q' = Q \setminus \{j\}$. By telescoping, we have
    \begin{align*}
        \mathbb{E}[f(\xi)] - \mathbb{E}[f(\xi) \mid \xi_{Q'} = 0] &= \sum_{i = 1}^{b-1} \left(\mathbb{E}[f(\xi) \mid \xi_{Q_{i-1}} = 0] - \mathbb{E}[f(\xi) \mid \xi_{Q_i} = 0]\right) \\
        &\stackrel{\eqref{eq:X_0}}{=} p \sum_{i = 1}^{b-1} \mathbb{E}[D_{q_i}f(\xi) \mid \xi_{Q_{i-1}} = 0],
    \end{align*}
    where $Q_i = \{q_1, \ldots, q_i\}$ (hence $Q_0 = \emptyset$). Applying \eqref{eq:exp_cond}, we get
    $$
        \mathbb{E}[D_{q_i}f(\xi) \mid Z_{Q_{i-1}}(\xi) = 0] \le \frac{\Inf_{q_i}(f)}{(1 - p)^{i-1}}.
    $$
    Using the fact that $\Inf_q(f) \le 2 \Inf_j(f)$ for every $q \in B_\ell$, this finally gives us
    \begin{equation} \label{eq:loss}          
        \mathbb{E}[f(\xi)] - \mathbb{E}[f(\xi) \mid \xi_{Q'} = 0] \le \frac{p}{(1 - p)^{b-2}} \sum_{i = 1}^{b-1} \Inf_{q_i}(f) \le \frac{2pb}{(1 - p)^{b-2}} \Inf_j(f) \le \frac{\Inf_j(f)}{8(1 - p)^{b-2}}.
    \end{equation}
    
    It remains to estimate $\mathbb{E}[f(\xi) \mid \xi_j = 1, \xi_{Q'} = 0]$. Write
    \begin{multline*}
        \mathbb{E}[f(\xi) \mid \xi_j = 1, \xi_{Q'} = 0] - \mathbb{E}[f(\xi)] = \\
        \left( \mathbb{E}[f(\xi) \mid \xi_j = 1, \xi_{Q'} = 0] - \mathbb{E}[f(\xi) \mid \xi_{Q'} = 0]\right) - \left(\mathbb{E}[f(\xi)] - \mathbb{E}[f(\xi) \mid \xi_{Q'} = 0] \right).
    \end{multline*}
    Applying Claim \ref{claim:key} and \eqref{eq:X_1} to bound the first term from below, and \eqref{eq:loss} to bound the second from above, we conclude
    $$
        \mathbb{E}[f(\xi) \mid \xi_j = 1, \xi_{Q'} = 0] - \mathbb{E}[f(\xi)] \ge (1 - p) \frac{\Inf_j(f)}{8 (1 - p)^{b-1}} - \frac{\Inf_j(f)}{8(1 - p)^{b-2}} \ge 0.
    $$
    This concludes the proof.
\end{proof}

Armed with Lemma \ref{lemma:main}, the proof of Theorem \ref{thm:main} is straightforward.

\begin{proof}[Proof of Theorem \ref{thm:main}]
    Let $p_0 = p_c(\cF)$. If $p_0\ge1/128$, then $b(\cF)\ge1\ge(128p_0)^{-1}$, since the first player wins with bias $1$. Otherwise, set
    $$
        b=\left\lfloor\frac{1}{64p_0}\right\rfloor\ge\frac{1}{128p_0}.
    $$
    We show that the first player wins with bias $b$. Let $S \subseteq X$ denote the set of elements claimed by the second player, and let $D \subseteq X$ denote discarded elements. The first player maintains that at the beginning of each round either $S \in \cF$ or
    $$
        \Pr[\xi \in \cF \mid S \subseteq \xi, D \cap \xi = \emptyset] \ge 1/2,
    $$
    where we think of $\xi \sim \mu_{p_0}$ as being the characteristic vector of a set. By the definition of $p_c(\cF)$, this holds at the beginning of the game. Suppose that it holds for some $S$ and $D$ obtained after playing the game for a number of rounds.

    If $S \in \cF$, then because $\cF$ is an increasing property the first player can simply offer any $b$ elements (or terminate the game, if there are fewer than $b$ available). Otherwise, let $U := X \setminus (D \cup S)$ and let $f: \{0,1\}^U \to \{0,1\}$ be defined as $f(T) = 1$ if and only if $S \cup T \in \cF$, where we identify a vector with its support $T\subseteq U$. By the inductive hypothesis, we have $\mathbb{E}[f(\xi)] \ge 1/2$ for $\xi \sim U_{p_0}$. As $\mathbb{E}_p[f(\xi)]$ is continuous and increasing in $p$, and $f(0) = 0$, there exists $0 < p \le p_0$ such that $\mathbb{E}_p[f(\xi)] = 1/2$ for $\xi \sim U_p$. 
    Since $bp\le bp_0\le1/64$, Lemma \ref{lemma:main} gives a set $Q \subseteq U$ of size $|Q| = b$. This is the set that the first player picks. Whichever element $q \in Q$ the second player chooses, the sets $S' = S \cup \{q\}$ and $D' = D \cup (Q \setminus \{q\})$ satisfy
    $$
        \Pr_{\xi\sim\mu_p}[\xi \in \cF \mid S' \subseteq \xi, D' \cap \xi = \emptyset] \ge 1/2.
    $$
    The unclaimed coordinates still have a product distribution. Since $\cF$ is increasing and $p\le p_0$, the same inequality holds with $p_0$ in place of $p$, preserving the invariant. As the game finishes after finitely many rounds, this implies that at some point we must have $S \in \cF$.
\end{proof}

\bibliographystyle{abbrv}
\bibliography{references}

\end{document}